\documentclass[reqno, 12pt]{amsart}
\usepackage{amssymb,latexsym,amscd,graphicx}
\usepackage{color}
\newtheorem{theorem}{Theorem}[section]
\newtheorem{lemma}[theorem]{Lemma}

\theoremstyle{definition}

\theoremstyle{definition}

\usepackage{geometry}
\definecolor{orange}{rgb}{1,0.5,0}
\begin{document}
\title{A new Proof of the Corona Problem  }
\author{Hunduma Legesse Geleta}
\address{Addis Ababa Univesity, College of natural and Computational Sciences, Department of Mathematics}
\email{hunduma.legesse@aau.edu.et}
\address{}
\email{}

\begin{abstract}
 The corona problem was motivated by the question of the density of the open unit disc  in the maximal ideal space of the algebra of bounded holomorphic functions on the unit disc. The corona problem connects operator theory, function theory, and geometry. It has been studied by different scholars in different contexts and found to be an important part of classical function theory, and of modern harmonic analysis. In this paper we give a new  proof of the corona theorem based on  Bezout's formulation of the corona problem. The main feature of this method sheds light on how to extend from complex one variable to complex several variables.  
\end{abstract}

\maketitle
\def\theequation{\thesection.\arabic{equation}}
\section{Introduction}

In Mathematics, the corona theorem is a result about the spectrum $\mathbb{S}$ of the bounded holomorphic functions on the unit disc $\mathbb{D}$, conjectured by Kakutani (1941)~~and proved by Lennart Carleson (1962)~~in his paper [3]. 
 The corona problem was motivated by the question of the density of the open unit disc $\mathbb{D}$ in the maximal ideal space of the algebra $H^\infty(\mathbb{D})$, of bounded holomorphic functions on $\mathbb{D}$ engaged a lot of Mathematician  to solve this problem in different ways and in different contexts. 
In the mid-1960s, Edgar Lee Stout [4] and Norman Alling [4] proved that the corona theorem remains true on a finitely-connected Riemann surfaces. By contrast, Brain Cole [4] gave an example of an infinitely connected Riemann surface on which the corona theorem fails. Cole's counter example was built by exploiting the connections between representing measures and uniform algebras. Around the same time, Kenneth Hoffman [4] showed that there is considerable analytic structure in the fiber of the maximal ideal space $\mathbb{H}^\infty(\mathbb{D})$.

In [4], Lars Hormander introduced a new method for studying the corona problem. He first constructed a preliminary nonanalytic function solution of Bezout's equation, and then correct it to get an analytic one by solving an appropriate in-homogeneous Riemann equations ($\bar{\partial}$- equation). The new method developed by Hormander allows to move from two pieces of corona data to $N$-pieces. But to do this Carleson used a clever trick and which was based on the Riemann mapping theorem, so there was no hope to generalize it to higher dimensions. Hormander's paper raised hopes that the corona theorem can be generalized to higher dimensions, but the hopes for an easy generalization were squashed by N. Varopoulos [4], who had shown that in the unit ball in two dimensional complex plane, the Carleson measure condition on the right hand side does not imply the existence of a bounded solution of $\bar{\partial}$- equation. As  a by-product, of Carleson's work, the Carleson measure was invented which itself is a very useful tool in modern function theory. It remains an open question whether there are versions of the corona theorem for every planar domain or for higher-dimensional domains. In 1979 T. Wolff presented a new proof of the corona theorem, which followed Hormander's approach with one critical difference. Wolff used a different condition for the existence of a bounded solution of the $\bar{\partial}$- equation $\partial w = G$, based on the "second order" Green's formula.

In this paper we are concerned with Bezout's formulation of the corona problem which is the topological equivalent statement about the density of the maximal ideal which was translated into the algebraic Bezout's formulation. Specifically, we prove the following: Let $f_1,  f_2, ...,f_N$ be bounded holomorphic functions on the unit disc $\mathbb{D}$in the complex plane $\mathbb {C}$ satisfying  $$\lvert f_1 \rvert+ \lvert f_2 \rvert+ ...+\lvert f_N\rvert > \delta > 0$$ on $\mathbb{D}$. Then there exist bounded analytic functions $g_1,  g_2, ..., g_N$ such that $$g_1(z)f_1(z)+g_2(z)f_2(z)+ ...+g_N(z)f_N(z)=1$$ on $\mathbb{D}$. 

We prove the corona theorem on the unit disc $\mathbb{D}$ and it can be easily extended to any bounded domain in the complex plane $\mathbb {C}.$ We also hope this method of proof can be extended to domains in complex several variables.

This paper is organized as follows: in section two we review the works of different scholars which are related to the corona problem specially in function theory. In section three we give a simple proof of corona problem. In section four we give summaries and conclusions.


\section{Preliminaries}
In this section we review some literature concerning the corona problem. Specifically we review the works of Carleson, Wolff, Hormander, and others as cited in [4]. We also collect some important properties of analytic functions which we may refer later on. 
The commutative Banach algebra and Hardy space $\mathbb{H}^\infty(\mathbb{D})$ consists of the bounded holomorphic functions on the unit disc $\mathbb{D}$ with supremum norm. It was conceived by I. M. Gelfand in his thesis in 1936 as cited in [4]. The spectrum $\mathbb{S}$ (the closed maximal ideals)contains $\mathbb{D}$ as an open subspace because for each $z$ in $\mathbb{D}$ there is a maximal ideal consisting of functions $f$ with $f(z)=0$. The subspace $\mathbb{D}$ cannot makeup the entire spectrum $\mathbb{S},$ essentially because the spectrum is a compact space and $\mathbb{D}$ is not. The complement of the closure of $\mathbb{D}$ in $\mathbb{S}$ was called the corona by Newman (1959), and the corona theorem states that the corona is empty, or in other words the open unit disc $\mathbb{D}$ is dense in the spectrum $\mathbb{S}.$ A more elementary formulation is that elements $f_1, f_2, ..., f_N$ generate the unit ideal of $H^\infty(\mathbb{D})$ if and only if there is some $\delta > 0$ such that $\lvert f_1 \rvert+ \lvert f_2 \rvert +...+ \lvert f_N \rvert > \delta > 0$ everywhere in the unit disc. Newman showed that the corona theorem can be reduced to an interpolation problem, which was then proved by Carleson [3]. In 1979 Thomas Wolff gave a simplified (but unpublished) proof of the corona theorem, as described in (Koosis 1980) and (Gamelin 1980). Cole later showed that this result cannot be extended to all open Riemann surfaces (Gamelin 1978).  Note that if one assumes the continuity up to the boundary in corona theorem, then the conclusion follows easily from the theory of commutative Banach algebra (Rudin 1991). 
The corona problem connects operator theory, function theory, and geometry. It is an important part of classical function theory, and of modern harmonic analysis. The corona problem was solved by Lennart Carleson in his paper [3]. In his proof he introduced the idea of Carleson measure which is used to control the length of certain curves in the disc that wind around the zeros of a bounded analytic function. The construction was very clever and quite involved and has proved to be useful in other areas of mathematics. The corona construction is widely regarded as one of the most difficult arguments in modern function theory as pointed out by Peter Jones, in [4].

 In 1979 T. Wolff presented a new proof of the corona theorem, which followed Hormander's approach with one critical difference. Wolff used a different condition for the existence of a bounded solution of the $\bar{\partial}$- equation $\partial w = G$, based on the "second order" Green's formula. That was crucial because, for the trivial non-analytic solution $$g_k =\frac{\bar f_k}{\sum_{j=1}^N \lvert f_j^2 \rvert}$$ of the Bezout's equation $$f_1g_1+f_2g_2+...+f_Ng_N=1$$

the right sides of the corresponding $\bar{\partial}$- equation satisfied this condition.

 \subsection*{Bezout's Formulation of the Corona Problem}
Let $f_1, f_2,...,f_N$ be bounded analytic functions on the unit disc $\mathbb{D}$ that satisfy $$\lvert f_1 \rvert+ \lvert f_2 \rvert +...+ \lvert f_N \rvert > \delta > 0$$ on $\mathbb{D}.$ Then do there exist bounded analytic functions $g_1, g_2, ..., g_N $ on $\mathbb{D}$ such that, $$f_1g_1+f_2g_2+...+f_Ng_N=1$$ on $\mathbb{D}$? Finding such bounded analytic functions could be very easy if $$G(z)=f_1(z)+f_2(z)+...+f_N(z)$$ was zero free on $\mathbb{D},$ since $$g_k(z)=\frac{1}{G(z)}, k= 1, 2, ...N$$ solve the corona problem of Bezout's formulation.

\section{Main Results}

In this section before proving the main theorem we state the following lemma from Real and Complex Analysis third edition by Walter Rudin. 
\begin{lemma}(Walter Rudin [7])
If $z_1, z_2, ..., z_N$ are complex numbers, then there is a subset $S$ of $ \{1, 2, ..., N\}$for which $$\lvert \sum_{k\in S}z_k\rvert \geq \frac{1}{\pi}\sum_{k=1}^N \lvert z_k\rvert.$$
\end{lemma}
We can modify lemma 3.1 for finite number of complex functions of one variable as in the following lemma.
\begin{lemma}
If $f_1(z), f_2(z),..., f_N(z)$ are complex functions of one variable on a domain $\Omega$, then there is a nonempty subset $S$ of $\{1, 2, ..., N\}$ for every $z\in\Omega$ ~~for which $$\lvert \sum_{k\in S}f_k(z)\rvert \geq \frac{1}{\pi}\sum_{k=1}^N \lvert f_k(z)\rvert.$$
\end{lemma}
\begin{proof}
Assume there does not exist such $S$ for every $z$ in the domain for which the inequality in the lemma 3.2 holds. Then for every nonempty subset $S$  of $\{1, 2, ..., N\}$ there is $z_0\in\Omega$ for which the inequality is violated. This implies that $f_1(z_0), f_2(z_0),..., f_N(z_0) $ are complex numbers for which the inequality is violated for every subset $S$ of $\{1, 2, ..., N\}.$
This contradicts the result in the lemma 3.1.\\
 Therefore, the assertion of lemma 3.2 holds true.
\end{proof}

\begin{theorem}
Let $f_1, f_2,...,f_N$ be in $H^\infty(\mathbb{D})$ such that $$\lvert f_1 \rvert+ \lvert f_2 \rvert +...+ \lvert f_N \rvert > \delta > 0$$ on $\mathbb{D}$. Then there exist a subset $S$ of $\{1, 2,..., N\}$ ~~for each $z$ in $\mathbb{D},$ such that $$\sum_{k\in S}f_k(z) \neq 0.$$ 
\end{theorem}

\begin{proof}
Suppose $f_1, f_2,...,f_N$ are in $H^\infty(\mathbb{D})$ such that $$\lvert f_1 \rvert+ \lvert f_2 \rvert +...+ \lvert f_N \rvert > \delta > 0$$ on $\mathbb{D}.$ Then by lemma 3.2, there is a subset $S$ of $\{1, 2, ..., N\}$ such that  $$\lvert \sum_{k\in S}f_k(z)\rvert \geq \frac{1}{\pi}\sum_{k=1}^N \lvert f_k(z)\rvert > \delta > 0 .$$
This implies that $$\lvert \sum_{k\in S}f_k(z)\rvert  > \delta > 0.$$ 
Thus, on $\mathbb{D}$ $$\sum_{k\in S}f_k(z) \neq 0.$$
\end{proof}
\bigskip

As given in [5] ~~the corona problem, when reinterpreted as problem in function theory asks, given $f_1, f_2,...,f_N$ in  $H^\infty(\mathbb{D})$ such that 
$$\lvert f_1 \rvert+ \lvert f_2 \rvert +...+ \lvert f_N \rvert > \delta > 0$$ on $\mathbb{D}$, there exist $g_1, g_2, ..., g_N $~ in ~~$ H^\infty(\mathbb{D}) $ ~such that, $$f_1g_1+f_2g_2+...+f_Ng_N=1$$ on $\mathbb{D}.$ Now we state this version of corona theorem and give a simple proof which follows directly from theorem 3.3.

\begin{theorem}
Let $f_1, f_2,...,f_N$ be in $H^\infty(\mathbb{D})$  satisfying Bezout's formulation of the corona problem. Then there are  $g_1, g_2, ...,_N$  in $H^\infty(\mathbb{D})$ such that
  $$f_1g_1+f_2g_2+...+f_Ng_N = 1.$$ 
\end{theorem}

\begin{proof}
Finding $g_1, g_2, ..., g_N $ in $ H^\infty(\mathbb{D}) $ is very easy if for each $z$ in $\mathbb{D}$, $$G(z)= f_1(z)+f_2(z)+...f_N(z)\neq 0.$$ In this case for each k, we put$$g_k(z)=\frac{1}{G(z)} ,$$ so that $g_k$ is analytic on $\mathbb{D}$ and $$f_1g_1+f_2g_2+...+f_Ng_N = 1.$$  Otherwise, by theorem 3.3 there exists  a subset $S$ of $\{1, 2,..., N\}$~~for each $z \in \mathbb{D} $ for which  $$G(z)= \sum_{k\in S}f_k(z) \neq 0.$$ 

In this case, we put $$g_k(z)=\frac{1}{G(z)} ,$$ for each $k \in S$ and   $g_k(z)=0$ for $k\notin S$ so that $g_k$ is bounded analytic functions on $\mathbb{D}$ for each $k\in \{1, 2, ..., N\}$. Moreover, $$\sum_{k=1}^N f_k(z)g_k(z)=\sum_{k\in S}f_k(z)g_k(z)+\sum_{k\notin S}f_k(z)g_k(z)=\sum_{k\in S}f_k(z)g_k(z)=1.$$ This completes the proof.
\end{proof}

\section{Concluding Remarks}The corona problem is an important part of classical function theory. The present paper offers the first positive result with simple proof with out using any of the methods developed to construct solutions of corona problem. It also sheds light on how to extend the result from complex one variable to complex several variables.

\end{document}